\documentclass[a4paper,10pt]{amsart}

\usepackage{geometry}
\usepackage{amsmath}
\usepackage{amssymb}
\usepackage{amsthm}
\usepackage[latin1]{inputenc}
\usepackage{eurosym}
\usepackage[dvips]{graphics}
\usepackage{graphicx}
\usepackage{epsfig}
\usepackage{enumitem}
\usepackage{mdwlist}

\usepackage{color}

\usepackage{ifthen}

\renewcommand{\div}{\operatorname{div}}

\newcommand{\Rr}{{\mathbb{R}}}

\newcommand{\Nn}{{\mathbb{N}}} 
\newcommand{\Tt}{{\mathbb{T}}}

\newcommand{\Fff}{{\mathcal{F}}}

\newtheorem{teo}{Theorem}
\newtheorem{df}{Definition}
\newtheorem{cor}{Corollary}
\newtheorem{lemma}{Lemma}
\newtheorem{pro}{Proposition}

\newtheorem{hyp}{Assumption}

\newtheorem{rem}{Remark}

\newcommand{\td}{\mathbb{T}^d}
\newcommand{\rd}{\mathbb{R}^d}

\begin{document}

\title{Short-time existence of solutions for mean-field games with congestion}

\author{Diogo A. Gomes}
\address[D. A. Gomes]{
	King Abdullah University of Science and Technology (KAUST), CSMSE Division , Thuwal 23955-6900. Saudi Arabia, and  
	KAUST SRI, Uncertainty Quantification Center in Computational Science and Engineering.}
\email{diogo.gomes@kaust.edu.sa}
\author{Vardan K. Voskanyan}
\address[V. Voskanyan]{
King Abdullah University of Science and Technology (KAUST), CSMSE Division , Thuwal 23955-6900. Saudi Arabia, and  
KAUST SRI, Uncertainty Quantification Center in Computational Science and Engineering.}
\email{vardan.voskanyan@kaust.edu.sa}

\keywords{Mean Field Game; Congestion models}
\subjclass[2010]{
	35J47, 
	35A01} 

\thanks{
	D. Gomes and V. Voskanyan were partially supported by KAUST baseline and start-up funds and 
	KAUST SRI, Center for Uncertainty Quantification in Computational Science and Engineering.  
}
\date{\today}

\maketitle

\begin{abstract}
We consider time-dependent mean-field games with congestion that are given by a system of a Hamilton-Jacobi
equation coupled with a Fokker-Planck equation. 
The congestion effects make the Hamilton-Jacobi equation singular. These models are motivated by crowd dynamics where agents have difficulty moving in high-density areas. 
Uniqueness of classical solutions for
this problem is well understood. However, existence of classical solutions,  was only known in very special cases - stationary problems with quadratic Hamiltonians and some time-dependent explicit examples.
Here, we prove short-time existence of $C^\infty$ solutions in the case of sub-quadratic Hamiltonians. 
\end{abstract}

\section{Introduction}
\label{intro}

Here, we study the time-dependent mean-field games with congestion given by the system
 \begin{equation}
\label{maineq}
\begin{cases}
-u_t-\Delta u+m^{\alpha}H_0\left(x, \frac{Du}{m^{\alpha}}\right)+b\cdot Du=V(x,m(x,t)),\\
m_t-\Delta m-\div(D_pH_0\left(x, \frac{Du}{m^{\alpha}}\right)m)-\div(bm)=0,\\
u(x, T)=\Psi(x),\ m(x, 0)=m_0(x).
\end{cases}
\end{equation}
Because we work in the spatially periodic setting, the variable $x$ takes values on the $d$-dimensional torus $\td$. The unknowns in \eqref{maineq} are  the functions $u\colon \td\times[0,T]\to\Rr$, and $m\colon \td\times[0,T]\to\Rr^+$. The functions
$H_0:\Tt^d\times \Rr^d\to\Rr$, 
$V:\Tt^d\times \Rr\to \Rr$, and 
$b, \Psi, m_0\colon \td\to\Rr$ are given $C^\infty$ functions with $m_0>0$. Moreover, $V(x,m)$ is increasing in $m$. 
Detailed hypotheses on  $H_0, b, V, \Psi$ and $ m_0$ are presented in Section \ref{assamp}.
A concrete Hamiltonian $H_0$ for which our results apply is the following:
for $\gamma\in(1,2)$, set $\frac{1}{\gamma}+\frac{1}{\gamma'}=1$.
Consider the  Lagrangian:  
\begin{equation}
\label{l0e}
L_0(x,v)=a(x)(1+|v|^2)^{\frac {\gamma'} {2} },
\end{equation}
where $a\in \mathcal{C}^{\infty}(\td),\ a> 0$.
Define $H_0$ as the Legendre transform of $L_0$:
\begin{equation}
\label{h0e}
 H_0(x,p)=\sup_v [-v\cdot p- L_0(x,v)]. 
\end{equation}

The uniqueness of solutions to \eqref{maineq}
was proven in \cite{LCDF} (see also \cite{GueU}) under Assumptions
\ref{uniq1}-\ref{uniq3} of Section \ref{assamp}.
Here, we prove the existence of smooth solutions for small terminal times
and sub-quadratic Hamiltonians:
\begin{teo}\label{main}
Under Assumptions \ref{psi}-\ref{uniq3}, (cf. Section \ref{assamp}), there exists a time $T_0>0$ such that for any terminal time $T\leq T_0$ there exists a $C^\infty$ solution $(u, m)$ to \eqref{maineq}, with $m>0$.
\end{teo}

Mean-field games have become 
an important research field since the seminal works of J-M. Lasry and P-L. Lions \cite{ll1, ll2, ll3}, and M. Huang, P. Caines and R. Malham\'e \cite{Caines1, Caines2}.
Diverse questions have been studied intensively,   
these include stationary mean-field games \cite{GM, GPM1, GPatVrt}, classical and weak solutions
for time-dependent problems, see, respectively, \cite{GPM2, GPM3, GPim1, GPim2} and  \cite{porretta, porretta2, cgbt}, 
finite state models \cite{GMS, GMS2, gomes2011, GF, GVW-Socio-Economic, GVW-dual}, 
extended mean-field games \cite{GVrt}, and obstacle problems \cite{GPat}. 
For a recent survey, see \cite{GS}.
Congestion problems were addressed initially by P-L. Lions \cite{LCDF}, who proved the uniqueness of smooth enough solutions.
Two alternative approaches to congestion problems are density constraints, 
introduced in \cite{San12, FrS}, and nonlinear mobilities, see \cite{BDFMW14}.
A recent existence result, see \cite{GMit}, regards the stationary congestion problem with quadratic costs. The existence of solutions in the time-dependent setting has not been established previously. 

Before proceeding, we briefly discuss  the motivation for \eqref{maineq}.
We consider a large population of agents on $\mathbb{T}^d$, whose statistical evolution over time is encoded in  an unknown probability density $m(x,t)$.  
Let  $(\Omega, \Fff_t, P)$ be a filtered probability space supporting a 
 $d$-dimensional Brownian motion 
  $W_t$. Let $\mathbb{E}$ be the expected value operator.
  Consider an agent whose location at time $t$ is  $x$.
The cost function 
for this agent, sometimes called value or utility function,
is
\[
u(x,t)=\inf\limits_v\mathbb{E}\int\limits_t^TL(X_s,v_s,m(X_s,s),s)ds+\Psi(X_T),
\]
where the trajectory $X$ controlled by the  dynamics
\[
dX_s=v_sds+\sqrt{2}dW_s,\ X_0=x,
\]
and the infimum is taken over bounded $\Fff_t$-progressively measurable
controls $v_s$. 
Here, $\Psi\colon\td\to\Rr$ is the terminal cost. The Lagrangian $L$ has the form
$$L(x,v,m,t)=m^{\alpha}L_0(x,v-b(x,t))+V(x, m).$$
Detailed assumptions on $L_0$ are given in the next Section. The constant $\alpha$ determines the strength of the congestion effects. These are encoded in the term $m^{\alpha}L_0(x,v-b(x,t))$ that makes it more expensive to move in regions of high density if the drift $v$ is substantially different from a reference vector field $b\colon \mathbb{T}^d\times [0,T]\to \mathbb{R}^d$. Finally, the function $V\colon \mathbb{T}^d\times \mathbb{R}^+_0\to \mathbb{R}$ accounts for other spatial preferences of the agents.

The Hamiltonian is the Legendre transform of $L$,
 given by
\[
H(x,p,m,t)=\sup\limits_v \{ -v\cdot p- L(x,v,m,t) \}=m^{\alpha}H_0\left(x, \frac{p}{m^{\alpha}}\right)+b(x,t)\cdot p,
\]
where $H_0$ is the Legendre transform of $L_0$.
Under standard assumptions regarding rationality and symmetry, the mean-field problem that models this setup is \eqref{maineq}. It comprises a system of
a second-order Hamilton-Jacobi equation for the value function $u$ coupled with a Kolmogorov-Fokker-Planck equation for
the density of agents $m$.

We conclude this introduction with the structure of the paper:
in Section \ref{assamp}, we state the main assumptions used in this manuscript. Afterwards, in Section \ref{estim},
we discuss various estimates that hold for arbitrary values of
the terminal time $T$. Then, in Section \ref{shortestim}, we present a new technique 
to address the short-time problem by controlling the growth of $\frac 1 m$.
Next, in Section \ref{srvf},
we establish further regularity for the solutions. 
Section \ref{exist} concludes the paper with
the proof of Theorem \ref{main}.  

\section{Assumptions}
\label{assamp}

Throughout the present manuscript, we work under several hypotheses that we state next.
Assumptions \ref{psi} and \ref{m0} concern the smoothness of the initial and terminal data, and the various functions in \eqref{maineq}. 
Here, we work with $C^\infty$ data to simplify the arguments. However, it would be possible to carry out the proofs with less regularity, and obtain 
the existence of solutions with $C^k$ regularity for $k$ large enough. 
Assumptions \ref{L0conv}-\ref{dhp} and \ref{uniq1}
are standard hypotheses in optimal control, viscosity 
solutions, and mean-field games.
They are stated explicitly for the convenience and clarity of the paper and do not result in a substantial loss of generality. 
A model Hamiltonian that satisfies those is 
\eqref{h0e}.
Assumptions \ref{alf} and \ref{H0sub}
are specific to the present problem and impose, respectively, 
a bound on the congestion exponent and 
 subquadratic growth for the Hamiltonian.
Subquadratic Hamiltonians correspond to
superquadratic Lagrangians. In the example \eqref{h0e},
this is reflected in
the condition
$\frac 1 \gamma+\frac 1 {\gamma'}=1$ satisfied by the exponent in 
\eqref{l0e}.
 Finally, Assumptions \ref{uniq2}
 and \ref{uniq3} are required for the uniqueness, see \cite{LCDF}.

\begin{hyp}
\label{psi} The terminal cost $\Psi\colon \mathbb{T}^d\to \Rr$, the reference velocity $b\colon \mathbb{T}^d\times [0,T]\to \Rr^d$ and the potential  $V\colon \mathbb{T}^d\times \Rr\to \Rr$ are $C^{\infty}$ functions, globally bounded with bounded derivatives of all orders.
\end{hyp}

\begin{hyp} \label{m0}
	The initial distribution
 $m_0\colon \mathbb{T}^d\to \Rr$ is a $C^{\infty}$ probability density: $\int_{\mathbb{T}^d}m_0(x)dx=1$. 
 Moreover, there exists $k_0>0,$ such that $m_0(x)\geq k_0$ for all $x\in\mathbb{T}^d.$
\end{hyp}

\begin{hyp}\label{L0conv}  The Lagrangian $L_0\colon \mathbb{T}\times \Rr^d$ is $C^\infty$, and the map
$$v\mapsto L_0(x, v)$$
is strictly convex for every $x\in\mathbb{T}^d.$
\end{hyp}

\begin{hyp} \label{L0pos}
  $L_0$ is positive: $L_0(x,v)\geq 0,\, \forall (x,v)\in \td\times\rd.$
\end{hyp}

\begin{hyp}\label{L0sup}  There exists conjugated powers $\gamma, \gamma'>1$, $\frac{1}{\gamma}+\frac{1}{\gamma'}=1$, and  constants $C_i,c_i>0,\ i=1,2$ such that
\[
C_1\frac{|v|^{\gamma'}}{\gamma'}-c_1
\leq L_0(x,v)\leq C\frac{|v|^{\gamma'}}{\gamma'}+c,\ \forall x\in\td, v\in\rd.
\]
\end{hyp}

\begin{rem}\label{rem1}
The definition of Legendre transform implies the convexity of $H_0$. Thus, we have
\begin{equation}\label{h0b}
H_0(x,p)-p\cdot D_pH_0(x,p)\leq H_0(x,0)=\sup\limits_v \big\{-L_0(x,v)\big\}\leq 0,
\end{equation}
using Assumption \ref{L0pos}.
\end{rem}

\begin{rem}\label{h0conv}
Under Assumption \ref{L0conv}, the Hamiltonian $H_0\colon\td\times\rd\to\Rr$ is $C^\infty$.
\end{rem}

\begin{rem}
\label{rem3}
Let $H_0(x,p)=\sup\limits_v \big\{-p\cdot v-L_0(x,v)\big\}$ be the Legendre transform of $L_0.$ Then, the Assumptions \ref{L0conv}-\ref{L0sup} imply
\[
C'_1\frac{|p|^{\gamma}}{\gamma}-c'_1\leq H_0(x,p)\leq C'_2\frac{|p|^{\gamma}}{\gamma}+c'_2,\quad \forall x\in\td, p\in\rd.
\]
\end{rem}

\begin{hyp}\label{H*0}  There exist positive constants $c,C>0$ such that 
\[
p\cdot D_pH_0(x,p)-H_0(x,p)\geq c |p|^{\gamma}-C.
\]
\end{hyp}

\begin{hyp}
	\label{dhp} There exists a constant $C$ such that
\[|D_pH_0(x,p)|\leq C |p|^{\gamma-1}+C.\]
\end{hyp}
\begin{rem}
\label{rem4}
By combining Remark \ref{rem3} with the Assumption \ref{H*0},
we conclude that there exist positive constants $c,C$ such that, for any $r>1$,
\[
c|p|^{\gamma}\leq H_0(x,p)+rp\cdot D_pH_0(x,p)+Cr.
\]
\end{rem}

\begin{hyp}\label{alf} 
The exponent $\alpha$ in the congestion term ($m^{\alpha}$) satisfies the inequality  $0\leq\alpha<\frac 2 {d-2}$.
\end{hyp}

\begin{hyp}\label{H0sub}
$H_0$ has sub-quadratic growth, i.e.  $\gamma<2.$ 
\end{hyp}

The next three assumptions are required for the uniqueness of solutions. 
\begin{hyp}
	\label{uniq1} The
	Hamiltonian $H_0$ is $C^\infty$, and the map
	$$p\mapsto H_0(x, p)$$ is strictly convex for every $x\in\mathbb{T}^d$, 
	that is, 
	$
	D^2_{pp}H_0(x,p)>0,\, \forall (x,p)\in\td\times\rd.
	$
\end{hyp}

\begin{rem}
	The previous Assumption implies that $H_0$ is uniformly convex on compacts, i.e., for any $R>0$, there exists $\theta_R>0$ such that
	$
	D^2_{pp}H_0(x,p)\geq\theta_R I,\, \forall (x,p)\in\td\times\rd \text{, with } |p|\leq R.
	$
\end{rem}

\begin{hyp}\label{uniq2} For $p\neq 0$, the following inequality holds:
	\[
	D_pH_0(x,p)\cdot p-H_0(x,p)> \frac{\alpha}{4}\ p^t\cdot D_{pp}H_0(x,p)\cdot p.
	\]
\end{hyp}

\begin{hyp}\label{uniq3}
	The potential $V\colon \td\times \Rr \to \Rr$ is strictly
	increasing in the second variable.
\end{hyp}

For the Hamiltonian $H_0$ given by \eqref{h0e},
Assumption \ref{uniq1} is satisfied for every $\gamma>1$.
A simple calculation shows that Assumption \ref{uniq2} holds if $\alpha<\frac{4(\gamma'-1)}{\gamma'}=\frac{4}{\gamma}$.
A potential $V$ for which 
Assumption \ref{uniq3} is valid is  $V(x,z)=\arctan(z)$.

\section{Estimates for arbitrary terminal time}
\label{estim}

The main result of this paper is the existence of smooth solutions to \eqref{maineq} for small terminal time $T$. Nevertheless, various estimates we need are valid for arbitrary $T$. We report those in this section. 

We begin with an auxiliary Lemma
\begin{lemma}\label{heat}
	For $0\leq \tau\leq T$, 
	and $\phi\in C^{\infty}(\mathbb{T}^d), \phi\geq 0$ with $\|\phi\|_{L^1(\mathbb{T}^d)}\leq 1$,
let $\rho$ be the solution to 
\begin{equation}
\label{heq}
\begin{cases}
\rho_t=\Delta \rho,\\
\rho(x, \tau)=\phi(x).
\end{cases}
\end{equation}
Denote by $2^*$ the Sobolev conjugate exponent of $2$, given by $\frac{1}{2^*}=\frac{1}{2}-\frac{1}{d}$. Then, for any $q$ with $1<q<\frac{2^*}{2}$, there exists a constant $C_q$ such that
\[
\|\rho\|_{L^1(L^q(dx), dt)}=\int\limits_{\tau}^T\left(\int_{\mathbb{T}^d}\rho^q dx\right)^{\frac{1}{q}}dt\leq C_q.
\]
\end{lemma}
\begin{proof}
By the maximum principle, $\rho\geq 0$. Furthermore, $\frac{d}{dt}\int_{\mathbb{T}^d}\rho(x,t)dx=0$. In particular, for any $t\geq\tau$
$\|\rho(\cdot,t)\|_{L^1(\mathbb{T}^d)}=\|\phi\|_{L^1(\mathbb{T}^d)}\leq 1.$ Multiplying the heat 
equation \eqref{heq} by $\rho^{\delta-1}$, for $0<\delta<1$, and integrating by parts, we get
\begin{equation}\label{drho}
c_{\delta}\int_{\tau}^T\int_{\mathbb{T}^d}|D(\rho^{\frac{\delta}{2}})|^2dxdt=
\frac{1}{\delta}\int_{\mathbb{T}^d}(\rho^{\delta}(x,T)-\rho^{\delta}(x,\tau))dx\leq \frac 1 {\delta},
\end{equation}
for any $\varepsilon>0,$ where $c_{\delta}=\frac{4(1-\delta)}{\delta^2}$.
 Here, we used Jensen's inequality to obtain
\[
0\leq\int_{\mathbb{T}^d}\rho^{\delta}(x,\tau)dx, \quad \int_{\mathbb{T}^d}\rho^{\delta}(x,T)dx\leq 1.
\]
From the Gagliardo-Nirenberg inequality,
\[
\|\rho^\frac{\delta }{2}(\cdot,t)\|_{L^{p_{\delta}}(\mathbb{T}^d)}\leq C \|D(\rho^{\frac{\delta}{2}})(\cdot,t)\|_{L^{2}(\mathbb{T}^d)}^{\delta} \|\rho^{\frac{\delta}{2}}(\cdot,t)\|_{L^{2}(\mathbb{T}^d)}^{1-\delta},
\]
where $\frac 1 {p_{\delta}}=(\frac 1 2- \frac 1 d)\delta+\frac{1-\delta}{2}$. Since $\|\rho^{\frac{\delta}{2}}(\cdot,t)\|_{L^{2}(\mathbb{T}^d)}\leq\|\rho(\cdot,t)\|_{L^{1}(\mathbb{T}^d)}^ \frac{\delta}{2}\leq 1,$ we have  $$\|\rho(\cdot,t)\|_{L^{\frac{\delta p_{\delta} }{2}}(\mathbb{T}^d)}\leq C \|D(\rho^{\frac{\delta}{2}})(\cdot,t)\|^2_{L^{2}(\mathbb{T}^d)}.$$ 
Finally, integrating the previous estimate in time and using \eqref{drho}, we obtain $\|\rho\|_{L^1(L^{\frac{\delta p_{\delta} }{2}}(dx), dt)}\leq C_{\delta}$. To end the proof, we observe that $p_{\delta} \to 2^*$ when $\delta\to 1$.
\end{proof}

\begin{pro}\label{ubound}
Under Assumptions \ref{psi}-\ref{L0pos}, there exists a constant $C:=C(\|V\|_{\infty}, \|\psi\|_{\infty},T)$ such that for any $C^\infty$ solution $(u, m)$ of \eqref{maineq}, we have  $\|m(\cdot, t)\|_{L^1(\mathbb{T}^d)}=1,\   \forall 0\leq t\leq T$, and, additionally,
$u\geq -C$.
\end{pro}
\begin{proof}
Integrating the second equation we have $\left(  \int_{\mathbb{T}^d}m(x,t)dx\right)_t=0$. Therefore,  $\int_{\mathbb{T}^d}m(x,t)dx=1$, for all $t\geq 0$. 
To prove the upper bound for $u$,
we apply the nonlinear adjoint method \cite{E3} (for further applications, see also \cite{T1}).
Let $\zeta$ be a  solution to
\begin{equation}
\label{adjoint}
\begin{cases}
\zeta_t-\div(D_pH_0\left(x, \frac{Du}{m^{\alpha}}\right)\zeta)-\div(b\zeta)=\Delta \zeta\\
\zeta(x,\tau)=\phi(x),
\end{cases}
\end{equation}
with $\phi\in C^\infty(\Tt^d)$, $\phi \geq 0$. We multiply the first equation in \eqref{maineq} by $\zeta$
and subtract \eqref {adjoint} multiplied by $u$. Then, we integrate by parts and gather
\[
-\left[\int_{\mathbb{T}^d} u\zeta dx \right]_t+\int_{\mathbb{T}^d}m^{\alpha}\left[H_0\left(x, \frac{Du}{m^{\alpha}}\right)- \frac{Du}{m^{\alpha}}D_pH_0\left(x, \frac{Du}{m^{\alpha}}\right)\right]\zeta dx=\int_{\mathbb{T}^d} V\zeta dx.
\]
Integrating on $[\tau, T]$ and using \eqref{h0b}, we obtain
\[
-\int_{\mathbb{T}^d} \Psi(x)\zeta(x,T)dx+\int_{\mathbb{T}^d} u(x,\tau)\phi(x)dx\geq\int_{\tau}^T\int_{\mathbb{T}^d}V\zeta dxdt.
\]
By the maximum principle, $\zeta \geq 0$, because $\phi\geq 0$. Integrating \eqref{adjoint} with respect to $x,$ we get $\int \zeta dx=\int \phi dx$.  This identity, together with the above inequality, yields
\[
\int_{\mathbb{T}^d} u(x,\tau)\phi(x)dx\geq-[(T-\tau)\|V\|_{\infty}+\|\psi\|_{\infty}]\|\phi\|_{L^1(\mathbb{T}^d)}.
\]
Since this estimate holds for every  $C^\infty$ $\phi \geq 0$, we obtain $-u(x,\tau)\leq (T-\tau)\|V\|_{\infty}+\|\psi\|_{\infty}$.
\end{proof}

\begin{pro}
Under Assumptions \ref{psi}-\ref{H*0},  there exists a constant $C:=C(\|V\|_{\infty}, \|\psi\|_{\infty},T)$ such that for any $C^\infty$ solution  $(u, m)$ to \eqref{maineq}, we have 
\begin{equation}
\label{AA1}
\int_0^T\int_{\mathbb{T}^d}
\frac{|Du|^{\gamma}}{m^{\bar{\alpha}}}
dx dt\leq C,
\end{equation}
and
\begin{equation}
\label{AA2}
\int_{\mathbb{T}^d} |u(x,t) |dx\leq C,\ t\in [0,T],
\end{equation}
where 
\begin{equation}
\label{abar}
\bar{\alpha}=(\gamma-1)\alpha<1.
\end{equation}
\end{pro}
\begin{proof}
We integrate the first equation in \eqref{maineq} with respect to $x$ and $t$. Then, we use the bounds on $u$ from the previous proposition, to get
\[
\int_t^T \int_{\mathbb{T}^d}  \left(m^{\alpha}H_0\left(x, \frac{Du}{m^{\alpha}}\right)+b\cdot  Du\right) dxds= \int_{\mathbb{T}^d}u(x,T)dx-\int_{\mathbb{T}^d}u(x,t)dx+\int_t^T \int_{\mathbb{T}^d}Vdxds\leq C.
\]
By Remark \ref{rem3}, 
 $|p|^{\gamma}\leq C(H_0(x,p)+b\cdot p)+C$. Accordingly, 
\begin{align*}
&\int_{\mathbb{T}^d}u(x,t)dx+\int_t^T\int_{\mathbb{T}^d}
\frac{|Du|^{\gamma}}{m^{\bar{\alpha}}}
dx dt\\&\qquad \leq C \int_t^T\int_{\mathbb{T}^d}  m^{\alpha}\left(H_0\left(x, \frac{Du}{m^{\alpha}}\right)+b\cdot  Du \right)dx ds+   C\int_t^T\int_{\mathbb{T}^d}  m^{\alpha}dxds \leq C.
\end{align*}
This inequality, combined with 
 the lower bound on $u$ of Proposition \ref{ubound}, yields \eqref{AA1}.
Moreover, since 
\[
\int u \leq C, 
\]
using again the lower bound on $u$ of Proposition \ref{ubound}, 
we obtain \eqref{AA2}.
\end{proof}

\begin{pro}
	\label{CC1}
Under Assumptions \ref{psi}-\ref{H*0},  there exists a constant $C:=C(\|V\|_{\infty}, \|\psi\|_{\infty},T)$ such that for any $C^\infty$ solution $(u, m)$ to \eqref{maineq} we have 
\[
\int_0^t\int_{\mathbb{T}^d}
|Du|^{\gamma}m^{1-\bar{\alpha}}
dx dt\leq C+C\int_0^t\int_{\mathbb{T}^d}m^{1+\alpha}dxdt,
\]
for all $0\leq t\leq T$, 
where $\bar \alpha$ is given by \eqref{abar}.
\end{pro}
\begin{proof}
We multiply the first equation in \eqref{maineq} by $m$ and subtract the second equation multiplied by $u$. Then, integration by parts yields:
\[
\begin{split}
&\int_0^t\int_{\mathbb{T}^d}m^{1+\alpha}\left[\frac{Du}{m^{\alpha}}D_pH_0
\left(x, \frac{Du}{m^{\alpha}}\right)-H_0\left(x, \frac{Du}{m^{\alpha}}\right)\right]dxdt=
-\int_0^t\int_{\mathbb{T}^d}Vmdx+\int_{\mathbb{T}^d}u(x,0)m_0(x)dx\\&-\int_{\mathbb{T}^d}u(x,t)m(x,t)dx\leq C+\|m_0\|_{\infty}\int_{\mathbb{T}^d}|u(x,0)|dx-\inf_x u(x,t)\leq C, 
\end{split}
\]
where the last inequality follows from lower bounds on $u$ from Proposition \ref{ubound}, and the bound on $\int_{\mathbb{T}^d} |u| dx$ in the previous Proposition.
The claim in the statement follows from Assumption \ref{H*0} by using the 
 inequality $pD_pH_0(x,p)-H_0(x,p)\geq c|p|^{\gamma}- C$, for some $c,C>0$.
\end{proof}

\begin{pro}
\label{uuper}
Under Assumptions \ref{psi}-\ref{dhp}, there exists a constant $C:=C(\|V\|_{\infty}, \|\psi\|_{\infty},T)$ such that for any $C^\infty$ solution  $(u, m)$  of \eqref{maineq}, we have 
\[
\int_{\mathbb{T}^d}m^{1+\alpha}(x,t)dx+
\int_0^t\int_{\mathbb{T}^d}
m^{\alpha-1}(x,s)|Dm(x,s)|^2
dx ds\leq C.
\]
\end{pro}
\begin{proof}
We begin by multiplying the second equation by $(\alpha+1)m^{\alpha}$. Next,  integrating by parts, we conclude
\begin{align}
\label{BB1}
&\left[\int_{\mathbb{T}^d}m^{1+\alpha}(x,t)dx\right]_t=-\alpha(1+\alpha)\int_{\mathbb{T}^d}m^{\alpha-1}\left[
|Dm|^2+mD_pH_0(x,\frac{Du}{m^{\alpha}})\cdot Dm+m b\cdot Dm\right]dx\\\notag&
\leq
-\alpha(1+\alpha)\int_{\mathbb{T}^d}\left[\frac 1 2 m^{\alpha-1}|Dm|^2- \left|D_pH_0\left(x,\frac{Du}{m^{\alpha}}\right)\right|^2m^{\alpha+1}-\|b\|^2_{\infty}m^{\alpha+1}\right]dx\\\notag&
\leq-\alpha(1+\alpha)\int_{\mathbb{T}^d}\left[\frac 1 2 m^{\alpha-1}|Dm|^2- C|Du|^{2(\gamma-1)}m^{\alpha+1-2\bar{\alpha}}-\|b\|^2_{\infty}m^{\alpha+1}\right]dx\\\notag&
\leq  -\alpha(1+\alpha)\int_{\mathbb{T}^d}\left[\frac 1 2 m^{\alpha-1}|Dm|^2-C|Du|^{\gamma}m^{1-\bar{\alpha}}-Cm^{\alpha+1}\right]dx,
\end{align}
where, in the last inequality, we have used Young's inequality:
\[
|p|^{2(\gamma-1)}m^{1+\alpha-2\bar{\alpha}}=\left(|p|^{\gamma}m^{1-\bar{\alpha}}\right)^{\frac{2(\gamma-1)}{\gamma}} \left(m^{1+\alpha}\right)^{\frac{2-\gamma}{\gamma}}\leq \frac{2(\gamma-1)}{\gamma} |p|^{\gamma}m^{1-\bar{\alpha}}+\frac{2-\gamma}{\gamma}m^{1+\alpha},
\]
and the definition of $\bar \alpha$ in \eqref{abar}.
Integrating \eqref{BB1} from $0$ to $t$
and using Proposition \ref{CC1},
 we conclude that
\begin{equation}
\label{malfa1}
\int_{\mathbb{T}^d}m^{1+\alpha}(x,t)dx+
\int_0^t\int_{\mathbb{T}^d}
m^{\alpha-1}(x,s)|Dm(x,s)|^2
dx ds\leq C+C\int_0^t\int_{\mathbb{T}^d}m^{1+\alpha}dxdt.
\end{equation}
In particular,
\[
\int_{\mathbb{T}^d}m^{1+\alpha}(x,t)dx\leq C+C\int_0^t\int_{\mathbb{T}^d}m^{1+\alpha}dxdt.
\]
Thus, by Gronwall's inequality, we have  $\int_{\mathbb{T}^d}m^{1+\alpha}(x,t)dx\leq C.$ This
estimate combined with \eqref{malfa1} yields $$\int_0^t\int_{\mathbb{T}^d}
m^{\alpha-1}(x,s)|Dm(x,s)|^2
dx ds\leq C.$$
\end{proof}

\begin{cor}
\label{BBB}
Under Assumptions \ref{psi}-\ref{dhp}, there exists a constant $C:=C(\|V\|_{\infty}, \|\psi\|_{\infty},T)$ such that 
\[
\int_0^t\int_{\mathbb{T}^d}
|Du|^{\gamma}m^{1-\bar{\alpha}}
dx dt\leq C
\]
\end{cor}
\begin{proof}
The Corollary follows by combining Proposition \ref{CC1} with Proposition 
\ref{uuper}.
\end{proof}

\begin{pro}
\label{uuperB}
Under Assumptions \ref{psi}-\ref{alf},  there exists a constant $C:=C(\|V\|_{\infty}, \|\psi\|_{\infty},T-t)$ such that, for any $C^\infty$ solution $(u,m)$ of \eqref{maineq},  $u\leq C$.
\end{pro}
\begin{proof}
Let $\rho$ be as in Lemma \ref{heat}. Multiplying the first equation in \eqref{maineq} by $\rho$, subtracting the equation for $\rho$ multiplied by $u$, and integrating by parts, we gather
\[
-\left[\int_{\mathbb{T}^d} u\rho dx \right]_t+\int_{\mathbb{T}^d}m^{\alpha}\left[H_0\left(x, \frac{Du}{m^{\alpha}}\right)+b\cdot \frac{Du}{m^{\alpha}}\right]\rho dx=\int_{\mathbb{T}^d} V\rho dx, 
\]
where we used the fact that $H_0(x,p)+b \cdot p$ is bounded by below as a consequence of Remark \ref{rem3}.
Hence, integrating in time, we conclude
\[
\int_{\mathbb{T}^d} u(x,\tau)\phi(x)dx\leq\int_{\tau}^T\int_{\mathbb{T}^d}V\rho dxdt+\int_{\mathbb{T}^d} \Psi(x)\rho(x,T)dx
+C\int_{\tau}^T\int_{\mathbb{T}^d}m^{\alpha}\rho dxdt.
\]
Using Holder's inequality and the bounds  $\int_{\mathbb{T}^d}\rho(x,t)dx\leq 1$, $\int_{\mathbb{T}^d}m^{1+\alpha}(x,t)dx\leq C$, we get
\begin{equation}
\label{upbin}
 u(x,\tau)
\leq (T-\tau)\|V\|_{\infty}+\|\Psi\|_{\infty}+C \|\rho\|_{L^1(L^{1+\alpha}(dx), dt)}.
\end{equation}
Because Assumption \ref{alf} holds,  $\alpha<\frac{2}{d-2}$. Consequently $1+\alpha<\frac{2^*}{2}$. Therefore, we can apply Lemma \ref{heat} to prove the result.
\end{proof}

%


\begin{cor}
\label{linf}Under Assumptions \ref{psi}-\ref{alf},  there exists a constant $C:=C(\|V\|_{\infty}, \|\psi\|_{\infty},T-t)$ such that for any $C^\infty$ solution $(u,m)$ of \eqref{maineq}  $\|u\|_{L^\infty(\Tt^d)}\leq C$.	
\end{cor}
\begin{proof}
The result follows by combining Propositions \ref{ubound} and  \ref{uuperB}.
\end{proof}
\begin{pro}
\label{D1mp}
Under Assumptions \ref{psi}-\ref{H0sub}, there exists constant $c_r, C_r=C_r(\alpha,T)>0$ that have polynomial growth in $r$, such that for any $C^\infty$ solution $(u,m)$ of \eqref{maineq}  and $r>1$
\[
\int_{\mathbb{T}^d}\frac{1}{m^r(x,t)}dx+c_r\int_0^t\int_{\mathbb{T}^d}\left| D\frac{1}{m^{r/2}}\right|^2dxdt+\int_0^t\int_{\mathbb{T}^d}\frac{|Du|^{\gamma}}{m^{r+\bar{\alpha}}}dxdt\leq C_r+C_r\int_0^t\int_{\mathbb{T}^d} \frac{1}{m^q}dxdt,\ \forall t,
\]
where $\bar \alpha$ is given by \eqref{abar} and  
\begin{equation}
\label{qforumula}
q=r+\frac{2\bar{\alpha}}{2-\gamma}.
\end{equation} 
\end{pro}

\begin{proof}
By adding a constant to $u_0$, we can assume, without loss of 
generality, that $u\leq -1$. Fix $r>1$.
We begin by
 multiplying the first equation in \eqref{main} by $\frac{1}{m^r}$, and adding
 it to the second equation multiplied by $r\frac{u}{m^{r+1}}$. After integrating by parts, we obtain
\[\begin{split}
&-\int_{\mathbb{T}^d}\left(  \frac{u}{m^{r}}\right)_tdx-\int_{\mathbb{T}^d}r(r+1)\frac{u|Dm|^2}{m^{r+2}} dx+\int_{\mathbb{T}^d}m^{\alpha}\frac{H_0+r\frac{Du}{m^{\alpha}} D_pH_0}{m^r}dx\\&
-\int_{\mathbb{T}^d}r(r+1)\frac{uD_pH_0Dm}{m^{r+1}}dx
+\int_{\mathbb{T}^d}\left[(r+1)  \frac{b\cdot Du}{m^{r}}-r(r+1) u\frac{b Dm}{m^{r+1}} -\frac{V}{m^r}\right]dx=0.
\end{split}
\]
We integrate the inequality in $t$.
For $m<1$, we have
\[
\frac{|Du|}{m^{r}}\leq \epsilon \frac{|Du|^{\gamma}}{m^{r+\bar{\alpha}}}+\frac{1}{m^r}+C_{r,\epsilon},
\]
whereas for $m\geq 1$
\[
\frac{|Du|}{m^{r}}\leq |Du|\leq  |Du|^\gamma m^{1-\bar \alpha}+C .
\]
Taking into account these estimates and the bound in 
Corollary \ref{BBB}, we get
\[
\int_0^t\int_{\Tt^d} \frac{|Du|}{m^{r}}\leq 
\int_0^t\int_{\Tt^d}
\frac{|Du|^{\gamma}}{m^{r+\bar{\alpha}}}+\frac{1}{m^r}+\tilde C_{r,\epsilon}.
\]
Then, we use the estimates:
\[
c\frac{|Du|^{\gamma}}{m^{r+\bar{\alpha}}}\leq m^{\alpha}\frac{H_0+r\frac{Du}{m^{\alpha}} D_pH_0}{m^r}+C\frac{r}{m^{r-\alpha}} \text{ (see Remark \ref{rem4}),} 
\]
and
\[
\frac{|Dm|}{m^{r+1}}\leq \epsilon\frac{|Dm|^2}{m^{r+2}}+C_{\epsilon}\frac{1}{m^r},
\] 
to get
\[
\begin{split}
&\int_{\mathbb{T}^d}\frac{1}{m^r(x,t)}dx+c_r\int_0^t\int_{\mathbb{T}^d}\left| D\frac{1}{m^{r/2}}\right|^2dxds+c_r\int_0^t\int_{\mathbb{T}^d}\frac{|Du|^{\gamma}}{m^{r+\bar{\alpha}}}dxds\leq 
\\&C_r\int_0^t\int_{\mathbb{T}^d}|Du|^{\gamma-1}\frac{|Dm|}{m^{r+\bar{\alpha}+1}}dxds+C_r\int_0^t\int_{\mathbb{T}^d}\frac{1}{m^{r}}dxds+C_r.
\end{split}
\]
The required estimate follows from the inequalities:
\[
|Du|^{\gamma-1}\frac{|Dm|}{m^{r+\bar{\alpha}+1}}\leq \epsilon\frac{|Du|^{\gamma}}{m^{r+\bar{\alpha}}}+\epsilon\frac{|Dm|^2}{m^{r+2}}+C_{\epsilon}\frac{1}{m^q},
\]
and
\[
\frac{1}{m^r}\leq \frac{1}{m^q}+1, 
\]
where $q$ is given by \eqref{qforumula}.
\end{proof}

\section{Short-Time Estimates}
\label{shortestim}

In this section, we establish estimates for $C^\infty$ solutions 
of \eqref{maineq} for small values of $T$. The key idea is to use the estimate in Proposition 
\ref{D1mp} to control the growth of $\frac 1 m$. Because $q>r$ in \eqref{qforumula}, 
we can only achieve bounds for small $T$. 
We begin with the following bound on $\frac 1 m$:

\begin{teo}
\label{mpshort}
Under Assumptions \ref{psi}-\ref{H0sub}, there exist  $r_0>0$, a time $t_1(r)>0$ and  constants $C=C(r,\gamma, \alpha)>0, \ \delta>0$, such that for any $C^\infty$ solution $(u, m)$ to \eqref{maineq}  and $r\geq r_0$:
\[
\int_{\mathbb{T}^d} \frac{1}{m^r(x,t)}dx\leq C\left[1+\frac{1}{(t_1-t)^{\delta}}\right],\ \forall t<t_1.
\]
\end{teo}
\begin{proof}
We choose $r_0$ sufficiently large such that  $\frac{2^*}{2}r=\frac{dr}{d-2}>q=r+\frac{2\bar{\alpha}}{2-\gamma}$,
for $r\geq r_0$.
Let $\lambda>0$ be such that $\frac{2^*}{2}r\lambda+r(1-\lambda)=q$, that is, $\lambda=\frac{\bar\alpha(d-2)}{(2-\gamma)r}<1,\text{ for } r\geq r_0$.
We set $\bar{\lambda}=\frac{2^*}{2}\lambda=\frac{\bar\alpha d}{(2-\gamma)r}$.
Provided $r_0$ is large enough, $\bar \lambda <1$ and $\beta=\frac{1-\bar{\lambda}}{1-\lambda}>1$,  for all $r\geq r_0$.
Then, using H\"older's and Young's inequalities, we obtain
 \[\begin{split}
\int_{\mathbb{T}^d} \frac{1}{m^{q}}dx\leq \left( \int_{\mathbb{T}^d} \frac{1}{m^{\frac{2^*}{2}r}}dx\right)^{\lambda} &  \left( \int_{\mathbb{T}^d} \frac{1}{m^r}dx \right)^{1-\lambda}= \left[ \left( \int_{\mathbb{T}^d} \frac{1}{m^{\frac{2^*}{2}r}}dx\right)^{2/2^*}\right]^{\bar{\lambda}}  \left[\left( \int_{\mathbb{T}^d} \frac{1}{m^r}dx \right)^{\beta}\right]^{1-\bar{\lambda}}\leq\\&
 \varepsilon  \bar{\lambda} \left( \int_{\mathbb{T}^d} \frac{1}{m^{\frac{2^*}{2}r}}dx\right)^{2/2^*}+\frac{1}{\varepsilon^\tau}(1-\bar{\lambda})\left( \int_{\mathbb{T}^d} \frac{1}{m^r}dx \right)^{\beta},
\end{split}\]
for any $\varepsilon>0$ and some exponent $\tau>0$.
From Sobolev's inequality,
\[
\int_{\mathbb{T}^d} \frac{|Dm|^2}{m^{r+2}}dx=\frac{4}{r^2}\int_{\mathbb{T}^d} \left|D\left(\frac{1}{m^{r/2}}\right)\right|^2dx\geq c\frac{4}{r^2}\left( \int_{\mathbb{T}^d} \frac{1}{m^{\frac{2^*}{2}r}}\right)^{2/2^*}dx-\frac{4}{r^2}\int_{\mathbb{T}^d} \frac{1}{m^{r}}dx.
\]
By combining Proposition \ref{D1mp} and the above inequalities
with the estimate
\[
\int_{\mathbb{T}^d} \frac{1}{m^{r}}dx\leq \varepsilon\left(\int_{\mathbb{T}^d} \frac{1}{m^{r}}dx\right)^{\beta}+C_{\varepsilon},\ \forall \varepsilon>0,
\]
we obtain
\begin{equation}
\int_{\mathbb{T}^d}\frac{1}{m^r(x,t)}dx\leq C+C\int_0^t\left(\int_{\mathbb{T}^d} \frac{1}{m^r}dx\right)^{\beta}dt,\ \forall t\in[0,T].
\end{equation}
Let $h(t)=\int_{\mathbb{T}^d} \frac{1}{m(x,t)^r}dx$ and $H(t)=\int\limits_0^th^{\beta}(s)ds$. Then, the previous inequality reads 
\[
h(t)\leq C_{r,\gamma,T}+  C_{r,\gamma,T}H(t).
\]
Thus,
\begin{equation}
\label{star}
\dot H(t)=h^{\beta}(t)\leq C_{r,\alpha, \gamma,T}(1+H(t))^{\beta}.
\end{equation}
Integrating \eqref{star} and taking into account that $H(0)=1$, we get
\[
(1+H(t))^{1-\beta}\geq 1-(\beta-1)C_{r,\gamma,T} t. 
\]
Accordingly, 
\[
H(t)\leq \frac{1}{\left[1-(\beta-1)C_{r,\gamma,T} t\right]^{\frac{1}{\beta-1}}},\text{ for all } t<t_1(r):=\frac{1}{(\beta-1)C_{r,\gamma,T}}.
\]
Consequently, 
\[
\int_{\mathbb{T}^d} \frac{1}{m(x,t)^r}dx=h(t)\leq C_{r,\gamma,T}+C_{r,\gamma,T}H(t)\leq C+\frac{C}{(t_1-t)^{\frac{1}{\beta-1}}},\ t<t_1.
\]
\end{proof}

\begin{cor}
\label{Dmpshort}
Suppose Assumptions \ref{psi}-\ref{H0sub} hold. Let $r_0$ and
$t_1(r)$
be as in Theorem \ref{mpshort}.  
For $r>r_0$, let $t\leq t_1(r)\equiv t_1$. Then, 
there exists a constant $C_r$
and $\delta_r$ 
such that for any $C^\infty$ solution (u,m) of \eqref{maineq} 
\begin{equation}
\label{om}
\int_0^t\int_{\mathbb{T}^d}\left| D\frac{1}{m^{r/2}}\right|^2dxdt\leq C_r+\frac{C_r}{(t_1-t)^{\delta_r}},\quad \forall \, t<t_1.
\end{equation}
\end{cor}

Iterating the estimates from Proposition \ref{D1mp}, 
we prove next bounds that are uniform in $r$. 
\begin{pro}
\label{1minf}
Under Assumptions \ref{psi}-\ref{H0sub}, there exist  $r_1>0$,  and  constants $C=C(r,\gamma, \alpha)>0, \ \beta_r>1,$ such that for any $C^\infty$ solution $(u, m)$ to \eqref{maineq}  and $r\geq r_1$:
\[
\left\| \frac{1}{m}\right\|_{L^{\infty}([0,t]\times \td)}\leq C_t\left(1+ \left\|\frac 1 m\right\|^{\beta_r}_{L^{\infty}([0,t],L^{r} (\Tt^d))}\right).
\]
\end{pro}

\begin{proof}
For  $r>1$, choose $\theta_n>0$ such that
\[
r^{n+1}+\delta=(1-\theta_n)r^n+\theta_n \frac{2^*}{2}r^{n+1},
\]
where $\delta=\frac{2\bar{\alpha}}{2-\gamma}$, that is $\theta_n=\frac{1-\frac 1 r + \frac{\delta}{r^{n+1}}}{\frac{2^*}{2}- \frac 1 r}>0$.
Set $\lambda_n=\frac{2^*}{2}\theta_n$ and $\beta_n=\frac{1-\theta_n}{1-\lambda_n}$.
Then, there exists $r_1>1$ such that for any $r\geq r_1$ and any $n\geq 1$, we have $\lambda_n<1$.
We fix a time $t$.
As in the previous proposition,
using a weighted Holder's inequality, we have
\[\begin{split}
&\int_{\mathbb{T}^d} \frac{1}{m^{r^{n+1}+\delta}}dx\leq
\left[ \left( \int_{\mathbb{T}^d} \frac{1}{m^{\frac{2^*}{2}r^{n+1}}}dx\right)^{2/2^*}\right]^{\lambda_n}  \left[\left( \int_{\mathbb{T}^d} \frac{1}{m^{r^n}}dx \right)^{\beta_n}\right]^{1-\lambda_n}\\&\leq
 \varepsilon  \lambda_n \left( \int_{\mathbb{T}^d} \frac{1}{m^{\frac{2^*}{2}r^{n+1}}}dx\right)^{2/2^*}+\frac{1}{\varepsilon^\tau}(1-\lambda_n)\left( \int_{\mathbb{T}^d} \frac{1}{m^{r^n}}dx \right)^{\beta_n},
\end{split}\]
where  $\varepsilon>0$ and $\tau>0$ is a suitable exponent. 
 On the other hand, Proposition \ref{D1mp} and Sobolev's inequality imply
\[
\int_{\mathbb{T}^d} \frac{1}{m^{r^{n+1}}(x,t)}dx+\int_0^t\left( \int_{\mathbb{T}^d} \frac{1}{m^{\frac{2^*}{2}r^{n+1}}(x,s)}dx\right)^{2/2^*} ds\leq C_{r^{n+1}}+C_{r^{n+1}}\int_0^t \int_{\mathbb{T}^d} \frac{1}{m^{r^{n+1}+\delta}(x,s)}dxds.
\]
From these two inequalities, we conclude:
\[
\int_{\mathbb{T}^d} \frac{1}{m^{r^{n+1}}(x,t)}dx\leq C_{r^{n+1}}+C_{r^{n+1}} \int_0^t \left( \int_{\mathbb{T}^d} \frac{1}{m^{r^n}(x,s)}dx \right)^{\beta_n}ds.
\]
Define $A_n(t)=\max_{[0,t]} \int_{\mathbb{T}^d} \frac{1}{m^{r^{n}}(x, \cdot)}dx$.
From the above estimate,
\[
1+A_{n+1}(t)\leq \max\{1,t\} C_{n} (1+A_n(t))^{\beta_n},
\]
where $C_n=O(r^{nk}),$ for some $k>1.$ Proceeding inductively, we get
\[
(1+A_{n+1}(t))^{\frac{1}{\beta_1\cdot \ldots \cdot \beta_n}} \leq C_t^{\sum_{i=1}^n \frac{1}{\beta_1\cdot\ldots\cdot\beta_i}}r^{\sum_{i=1}^n \frac{ik}{\beta_1\cdot\ldots\cdot\beta_i}}(1+A_1).
\]
Since 
\[
\beta_n =\frac{1-\theta_n}{1-\lambda_n}=r\left(1+ \frac{(\frac{2^*}{2} r -1)\delta}{r^{n+1}(\frac{2^*}{2}- 1-\frac{2^*}{2} \frac{\delta}{r^{n}})}\right):=r(1+q_n),
\]
where $q_n=O(r^{-n})>0$, 
the series  $\sum_{i=1}^{\infty} \frac{ik}{\beta_1\cdot\ldots\cdot\beta_i},$ $\sum_{i=1}^{\infty} \frac{1}{\beta_1\cdot\ldots\cdot\beta_i}$, and the infinite product $\prod_{i=1}^{\infty} (1+q_i)$ converge. From this, we
obtain
\[
\left\|\frac 1 m\right\|_{L^{\infty}([0,t],L^{r^{n+1}} (\Tt^d))}\leq C_t\left(1+ \left\|\frac 1 m\right\|^{\beta_r}_{L^{\infty}([0,t],L^{r} (\Tt^d))}\right),
\]
for some constants $C_t>0$ and $\beta_r=\prod_{i=1}^{\infty} (1+q_i)>1$ that do not depend on the solution. Sending $n\to \infty$ yields the result.
\end{proof}

The results of Theorem \ref{mpshort}, Proposition \ref{1minf} and Corollary \ref{Dmpshort} prove the following:
\begin{teo}\label{1minf2}
Under Assumptions \ref{psi}-\ref{H0sub}, there exist a time $T_0>0$ and  constants $C=C(\gamma, \alpha)>0,$ such that for any $C^\infty$ solution $(u, m)$ to \eqref{maineq}:
\[
\left\|\frac 1 {m}\right\|_{L^{\infty}([0, T_0]\times\td)}\leq C.
\]
\end{teo}

\section{Short-time regularity of the value function}
\label{srvf}

Building upon the results in the previous Section, we prove next further
regularity for the solutions of \eqref{maineq}.

\begin{lemma}\label{FokPlank}
Let $w:\Tt^d\times [0,T]\to \Rr$ be a non-negative solution
of the Fokker-Planck equation:
 \begin{equation}
\label{Fok-Plank}
w_t(x,t)-\Delta w-\div(g(x,t)w(x,t))=0,
\end{equation}
with $w(x, 0)=m_0(x)$.
Assume that for some $p_0> d$, every $r>1$ and some constants $C_r>0$
the drift $g$ satisfies
 $\left\|g\right\|_{L^r([0, T_0], L^{p_0}(\td))} \leq C_{r}$. Then, there exist constants $C_q$ such that $\left\|w\right\|_{L^{\infty}([0, T_0], L^{q}(\td))} \leq C_{q},$ for all $q>1.$
\end{lemma}
\begin{proof}
Multiplying \eqref{Fok-Plank} by $qw^{q-1}$ and integrating by parts, we get
\begin{align*}
\int_{\td}w^{q}(x,t)dx-&\int_{\td}w^{q}(x,0)dx+
q(q-1)\int_0^t\int_{\td}
w^{q-2}(x,s)|Dw(x,s)|^2
dx ds\\&=q(q-1)\int_0^t\int_{\td}w^{q-1} g\cdot Dwdxds.
\end{align*}
From this, using Cauchy inequality, we have the estimate
\begin{equation}\label{maininqw}
\int_{\td}w^{q}(x,t)dx+\int_0^t\int_{\td}
|Dw^{\frac q 2}(x,s)|^2
dx ds\leq C_q+C_q\int_0^t\int_{\td}|g|^2w^{q}dxds.
\end{equation}
The previous bound together with Sobolev's inequality implies
\[\begin{split}
&\int_0^t\|w^q(\cdot, s)\|_{L^{\frac {2^*} 2}({\td})}ds\leq C_q\int_0^t\int_{\td}(1+|g|^2)w^{q}dxds \\&\leq \tilde C_q\left(1+\int_0^t\|1+|g(\cdot, s)|^2\|^r_{L^{p_1}({\td})}ds+\int_0^t\|w^q(\cdot, s)\|^{r'}_{L^{p_1'}({\td})}ds\right),
\end{split}
\]
for any $r>1$, where $p_1=\frac {p_0} 2$ and the conjugate powers $r',p_1'$ satisfy  $\frac 1 r+ \frac 1 {r'}=1$, $\frac 1 {p_1}+ \frac 1 {p_1'}=1$. Recall that $\|w(\cdot, t)\|_{L^{1}({\td})}=1$. Moreover, because $p_1> \frac d 2$, we have $p_1'< \frac {2^*} 2$. Therefore,  by interpolation,
\[
\|w(\cdot, s)\|_{L^{qp_1'}({\td})}\leq \|w(\cdot, s)\|^{\theta}_{L^{q\frac {2^*} 2}({\td})} \|w(\cdot, s)\|^{1-\theta}_{L^{1}({\td})}=\|w(\cdot, s)\|^{\theta}_{L^{q\frac {2^*} 2}({\td})},
\]
for some $\theta<1$. Hence, 
$
\|w^q(\cdot, s)\|_{L^{p_1'}({\td})}\leq \|w^q(\cdot, s)\|^{\theta}_{L^{\frac {2^*} 2}({\td})}
$. By combining these bounds, we have the estimate
\[
\int_0^t\|w^q(\cdot, s)\|_{L^{\frac {2^*} 2}({\td})}ds\leq C_q\left(\int_0^t\|1+|g(\cdot, s)|^2\|^r_{L^{p_1}({\td})}ds+\int_0^t\|w^q(\cdot, s)\|^{r'\theta}_{L^{\frac {2^*} 2}({\td})}ds\right).
\]
Finally, by
choosing $r$ large enough so that $r'\theta<1$, we obtain $\int_0^t\|w^q(\cdot, s)\|_{L^{\frac {2^*} 2}({\td})}ds\leq C_q$. To end the proof, we observe that, from \eqref{maininqw},
it follows that, for any  $q>1$,
there exists $C_q$ such that 
$\int_{\td}w^q(x,t)dx\leq C_q$, for any
 $t\in [0,T]$. 
\end{proof}

The next Lemma uses the Gagliardo Niremberg theorem to obtain additional 
regularity. This is a critical point where we use the hypothesis that 
$H$ is subquadratic. 
\begin{lemma}\label{uLrp}
Under Assumptions \ref{psi}-\ref{H0sub}, there exist a time $T_0>0$ and  constants $C_p, C_{r,p}=C(\gamma, \alpha,r, p)>0,$ such that for any $C^\infty$ solution $(u, m)$ to \eqref{maineq}:
\[
\left\|u_t\right\|_{L^r(0, T_0; L^p(\td))}+ 
\left\|D^2 u\right\|_{L^r(0, T_0; L^p(\td))} \leq C_{r,p},\qquad \forall\, 1<r,p<+\infty.
\]
\[
\left\|m\right\|_{L^{\infty}(0, T_0; L^p(\td))}\leq C_p, 
\]
Furthermore,
\[
\left\|D u\right\|_{L^{p\gamma}(\td\times [0,T_0])}\leq C,\qquad \forall p>1.
\]
\end{lemma}
\begin{proof}
We choose $T_0$ as in Theorem \ref{1minf2}. By the Gagliardo-Nirenberg interpolation inequality and Corollary \ref{linf},
taking into account that $\gamma<2$, 
\[
\|Du(\cdot, t)\|_{L^{p\gamma}(\td)}\leq c_{p, d} \|D^2u(\cdot, t)\|_{L^p(\td)}^{\frac 1 2}\|u(\cdot, t)\|_{L^{\infty}(\td)}^{\frac 1 2}\leq C\|D^2u(\cdot, t)\|_{L^p(\td)}^{\frac 1 2}.
\]
For this reason, we have the bound
\[\begin{split}
&\left\||D u|^{\gamma}\right\|_{L^r(0, T_0; L^p(\td))}=\left(\int_0^{T_0}\|Du(\cdot, t)\|_{L^{p\gamma}(\td)}^{r\gamma}dt\right)^{\frac 1 r}\leq \\&C \left(\int_0^{T_0}\|D^2u(\cdot, t)\|_{L^{p}(\td)}^{\frac{r\gamma}{2}}dt\right)^{\frac 1 r}\leq C \left\|D^2 u\right\|^{\frac{\gamma}{2}}_{L^r(0, T_0; L^p(\td))},
\end{split}
\]
where we used again that $\frac{\gamma}{2}< 1$.
Then, from Theorem \ref{1minf2} and standard regularity results for the heat equation (see, for instance, \cite{Lad}), we have 
\begin{align*}
\left\|u_t\right\|_{L^r(0, T_0; L^p(\td))}+ 
\left\|D^2 u\right\|_{L^r(0, T_0; L^p(\td))} &\leq C\left\||D u|^{\gamma}\right\|_{L^r(0, T_0; L^p(\td))} + C\left\|m^{\alpha}\right\|_{L^r(0, T_0; L^p(\td))}  \\& \leq C\left\|D^2 u\right\|^{\frac{\gamma}{2}}_{L^r(0, T_0; L^p(\td))}+ C\left\|m^{\alpha}\right\|_{L^r(0, T_0; L^p(\td))}.
\end{align*}
Since $\frac{\gamma}{2}< 1$, we obtain \[
\left\|u_t\right\|_{L^r(0, T_0; L^p(\td))}
C\left\|D^2 u\right\|_{L^r(0, T_0; L^p(\td))} \leq C\left\|m^{\alpha}\right\|_{L^r(0, T_0; L^p(\td))}.
\]

The above arguments also imply  $$\left\|D u\right\|_{L^{r\gamma}(0, T_0; L^{p\gamma}(\td))}\leq C\left\|D^2 u\right\|^{\frac 1 2}_{L^r(0, T_0; L^p(\td))} \leq C+C\left\|m^{\alpha}\right\|^{\frac 1 2}_{L^r(0, T_0; L^p(\td))}\leq C_{r,p},$$ for all $r, p>1$.

We recall that $\alpha<\frac{d-2}{d}$. By Proposition \ref{uuper}, we have $\left\|m^{\alpha+1}\right\|_{L^{\infty}(0, T_0; L^1(\td))}\leq C$. 
Hence, for some $p_0>\frac d 2$,
$\left\|m^{\alpha}\right\|_{L^{\infty}(0, T_0; L^{p_0}(\td))}\leq C$. Consequently, $\left\|D u\right\|_{L^{r}(0, T_0; L^{p_0\gamma}(\td))}\leq C_r$, for any $r>1.$ Note that $p_0\gamma/ (\gamma-1)>\frac d 2 \gamma/ (\gamma-1)>d$. Therefore, $D_pH$ is bounded in $L^{r}(0, T_0; L^{p}(\td))$ for some $p>d$ and any $r>1$. Finally, Lemma \ref{FokPlank} implies that $m$ is bounded in $L^{\infty}(0, T_0; L^{q}(\td))$ for any $q>1.$
\end{proof}

\begin{lemma}\label{ulip}
Under Assumptions \ref{psi}-\ref{H0sub}, there exist a time $T_0>0$ and a constant $C>0,$ such that, for any $C^\infty$ solution $(u, m)$ to \eqref{maineq}, we have $\|Du\|_{L^{\infty}(\td\times [0,T_0])} \leq C.$
\end{lemma}
\begin{proof}
From Theorem \ref{1minf2} and Lemma \ref{uLrp}, it follows that the equation for $u$ can be written as
\[
\begin{cases}
u_t+\Delta u=f,\\
u(x,T_0)=\Psi(x),
\end{cases}
\]
where $f\in L^r(\td\times [0,T_0])$ for every $r>1.$ From this,
reasoning as in \cite{GPM2}, we obtain
 \[
 \|Du\|_{L^{\infty}(\td\times [0,T_0])} \leq C.
 \]

\end{proof}

\begin{lemma}\label{mlip}
Under Assumptions \ref{psi}-\ref{H0sub}, there exist a time $T_0>0$ and a constant $C>0$ such that, for any $C^\infty$ solution $(u, m)$ to \eqref{maineq}, we have $\|m\|_{L^{\infty}(\td\times [0,T_0])}, \|Dm\|_{L^{\infty}(\td\times [0,T_0])} \leq C.$
\end{lemma}
\begin{proof}
From the estimates  in Theorem \ref{1minf2} and Lemmas \ref{uLrp} and \ref{ulip},
 it follows that
 for suitable functions
 $a$, and $c$, bounded in $L^p(\td\times [0,T_0])$ for every $p>1$,
 the equation for $m$ can be written as
\[
\begin{cases}
-m_t(x,t)+\Delta m(x,t)=a(x,t)\cdot Dm(x,t)+c(x,t)m(x,t),\\
m(x,0)=m_0(x).
\end{cases}
\]
Let $w=\ln m$ then
\[
\begin{cases}
-w_t(x,t)+\Delta w(x,t) +|Dw|^2-a(x,t)\cdot Dw(x,t)=c(x,t),\\
w(x,0)=\ln m_0(x).
\end{cases}
\]
The adjoint method, applied as in \cite{GPatVrt}, yields  $\|w\|_{L^{\infty}(\td\times [0,T_0])}, \|Dw\|_{L^{\infty}(\td\times [0,T_0])} \leq C$. These estimates imply the result.
\end{proof}

\begin{teo}\label{bootstrapp}
Under Assumptions \ref{psi}-\ref{H0sub}, there exist a time $T_0>0$ and constants $C_{k,l,p}>0$,  for $k,l\in \mathbb{N},\, p>1$ such that, for any $C^\infty$ solution $(u, m)$ to \eqref{maineq}, 
we have
$\|D^k_tD^l_xu\|_{L^{p}(\td\times [0,T_0])} \leq C_{k,l,p}$.
\end{teo}
\begin{proof}
The result follows by a simple bootstrapping argument. 
As a starting point, we use the
 regularity given by Theorem \ref{1minf2} and Lemmas \ref{uLrp}, \ref{ulip}, and \ref{mlip}. 
 Then, the the Theorem is proven by repeatedly using the parabolic regularity on the equations for $u$, $m$, and their derivatives.
\end{proof}

\section{Existence of solutions}
\label{exist}

To establish the existence of solutions, we will use the continuation method. For that, we introduce the problem
\begin{equation}
	\label{maineqlambda}
	\begin{cases}
		-u_t-\Delta u+m^{\alpha}H_{\lambda}\left(x, \frac{Du}{m^{\alpha}}\right)+b_\lambda\cdot Du=V_\lambda(x,m(x,t)),\\
		m_t-\Delta m-\div(D_pH_{\lambda}\left(x, \frac{Du}{m^{\alpha}}\right)m)-\div(b_\lambda m)=0,\\
		u(x, T)=\Psi_{\lambda}(x),\ m(x, 0)=m_{\lambda}(x), 
	\end{cases}
\end{equation}
where $0\leq \lambda\leq 1$, 
$H_{\lambda}(x,p)=(1-\lambda)H_0+\lambda (1+|p|^2)^{\frac {\gamma} 2}$, $b_{\lambda}=(1-\lambda)b$, $V_{\lambda}=(1-\lambda)V+\lambda \arctan(m)$, $\Psi_{\lambda}=(1-\lambda)\Psi, m_{\lambda}=(1-\lambda)m_0+\lambda$.
The terminal time $T$ satisfies
$T\in \mathcal{T}=[0,T_0]$, where $T_0$ is as in Theorem \ref{bootstrapp}.

When $\lambda=1$, \eqref{maineqlambda} has a unique solution, namely
$u\equiv (1-\frac{\pi} 4) t$, $m\equiv 1$. We will prove that the set
$\Lambda$ of values $0\leq \lambda\leq 1$ for which \eqref{maineqlambda} admits
a solution is relatively open and closed. Therefore, $\Lambda=[0,1]$ and, in 
particular, \eqref{main} admits a solution. 

For  $k\geq -1$, we set $F^{k}(\mathcal{T}; \td)=\cap_{2k_1+k_2= k} H^{k_1}(\mathcal{T}; H^{k_2}(\td))$, where the intersection is taken over all integers $k_1\geq 0, k_2\geq -1$. 
The space $F^{k}(\mathcal{T}; \td)$ is a Banach space endowed with the norm
\[
\|f\|_{F^{k}(\mathcal{T}; \td)}=\sum_{2k_1+k_2= k}\left\| f \right\|_{H^{k_1}(\mathcal{T}; H^{k_2}(\td))}. 
\]
Moreover, there exists $\tilde k_d$, depending only on the dimension $d$, such that for $k\geq \tilde k$, the space $F^{k-2}$ is an algebra. 
Let $k\geq \tilde k_d$, and
consider the operator 
\[
\mathcal{M}_{\lambda}\colon F^{k}(\mathcal{T}; \td)\times F^{k}(\mathcal{T}; \td)\to
F^{k-2}(\mathcal{T}; \td)\times F^{k-2}(\mathcal{T}; \td)\times H^{k-1}(\td)\times H^{k-1}(\td)
\]
given by
\[
\mathcal{M}_{\lambda}\left[\begin{array}{c}
u\\
m
\end{array}\right]=
\left[\begin{array}{c}
m_t-\Delta m-\div(D_pH_{\lambda}\left(x, \frac{Du}{m^{\alpha}}\right)m)+\div(b_{\lambda} m)\\
u_t+\Delta u-m^{\alpha}H_{\lambda}\left(x, \frac{Du}{m^{\alpha}}\right)-b_{\lambda}\cdot Du+V_{\lambda}(x, m)\\
m(x,0)-m_{\lambda}(x)\\
u(x,T)-\Psi_{\lambda}(x)
\end{array}\right].
\]
Then, \eqref{maineqlambda} is equivalent to 
\begin{equation}
\label{eqlalt}
\mathcal{M}_{\lambda}\left[\begin{array}{c}
u\\
m
\end{array}\right]=0, 
\end{equation}
and \eqref{main} then reads as $\mathcal{M}_0\left[\begin{array}{c}
u\\
m
\end{array}\right]=0$ . Moreover, as we remarked before, 
$\mathcal{M}_1\left[\begin{array}{c}
u\\
m
\end{array}\right]=0,$ has only the trivial solution $u\equiv (1-\frac{\pi} 4) t$, $m\equiv 1$.
We consider the linearized operator $\mathcal{L}$:
\[
\mathcal{L}_{\lambda}\left[\begin{array}{c}
v\\
f
\end{array}\right]=\lim_{\varepsilon\to 0}\frac{ \mathcal{M}_{\lambda}\left[\begin{array}{c}
	u+\varepsilon v\\
	m+\varepsilon f
	\end{array}\right] -\mathcal{M}_{\lambda}\left[\begin{array}{c}
	u\\
	m
	\end{array}\right]}{\varepsilon}=
\]
\vskip0.3cm
\[
\left[\begin{array}{c}
f_t-\Delta f-\div\left[D_pH_{\lambda}\left(x, \frac{Du}{m^{\alpha}}\right)f+m^{1-\alpha}D^2_{pp}H_{\lambda}\left(x, \frac{Du}{m^{\alpha}}\right)\cdot Dv
-\alpha f D^2_{pp}H_{\lambda}\left(x, \frac{Du}{m^{\alpha}}\right)\cdot  \frac{Du}{m^{\alpha}} +b_{\lambda}f\right]\\
\\
v_t+\Delta v-\alpha m^{\alpha-1}f \left( H_{\lambda}\left(x, \frac{Du}{m^{\alpha}}\right) -  \frac{Du}{m^{\alpha}} D_pH_{\lambda}\left(x, \frac{Du}{m^{\alpha}}\right)\right)-D_pH_{\lambda}\left(x, \frac{Du}{m^{\alpha}}\right)Dv -b_{\lambda}\cdot Dv+D_zV_{\lambda}f\\
\\
f(x,0)
\\
\\
v(x,T)
\end{array}\right].
\]
Note that $\mathcal{L}_{\lambda}:F^{k}(\mathcal{T}; \td)\times F^{k}(\mathcal{T}; \td)\to
F^{k-2}(\mathcal{T}; \td)\times F^{k-2}(\mathcal{T}; \td)\times H^{k-1}(\td)\times H^{k-1}(\td)$, for all $k$ large enough. However, if $u$ and $m$
are $C^\infty$ solutions to \eqref{maineqlambda}, then  $\mathcal{L}_{\lambda}$ admits
a unique extension as bounded linear operator
$\mathcal{L}_{\lambda}:F^{k}(\mathcal{T}; \td)\times F^{k}(\mathcal{T}; \td)\to
F^{k-2}(\mathcal{T}; \td)\times F^{k-2}(\mathcal{T}; \td)\times H^{k-1}(\td)\times H^{k-1}(\td)$, for all $k\geq 1$.

The form $\langle \cdot, \cdot \rangle$ denotes the scalar product on $L^2(\td)$. To apply the inverse function theorem, we need to prove that the linear operator  $\mathcal{L}_{\lambda}$ is invertible. For this, we 
begin by showing that the equation $\mathcal{L}_{\lambda} w= W$ has a unique weak solution in the sense of the following definition:
 \begin{df}
 	\label{weaksolu}
For $h,g \in L^2(0,T_0; L^2(\td)), A,B \in L^2(\td)$, set
 \[
W(x,t)=\left[\begin{array}{c}
h(x,t)\\
g(x,t)\\
A(x)
\\
B(x)
\end{array}\right].
\]
A function
$w=\left[\begin{array}{c}
v\\
f
\end{array}\right]$, with
\begin{equation}\label{weaksol}
v,f \in L^2(0,T_0; H^1(\td)) \text{ and } v_t,f_t\in L^2(0,T_0; H^{-1}(\td)), \text{ that is } v,f\in F^1(\mathcal{T}; \td),
\end{equation}
is a weak solution of  $\mathcal{L}_{\lambda} w= W$ if:
\begin{enumerate}
\item for any $\bar v, \bar f \in H^1(\td)$ and for a.e. $t, 0\leq t\leq T_0$  we have
\begin{equation}\label{weak}
\begin{cases}
\langle f_t, \bar f \rangle +\left\langle Df + D_pH_{\lambda}f+m^{1-\alpha}D^2_{pp}H_{\lambda}\cdot Dv
-\alpha f D^2_{pp}H_{\lambda}\cdot Q +f b_{\lambda}, D\bar{f} \right\rangle= \langle h , \bar f \rangle\\

\langle v_t, \bar v \rangle -\langle D v , D\bar v \rangle-\left\langle \alpha m^{\alpha-1}f \left( H_{\lambda} -Q\cdot D_pH_{\lambda}\right)+D_pH_{\lambda} \cdot Dv +b_{\lambda}\cdot Dv-D_zV_{\lambda}f,\bar{v} \right\rangle= \langle g , \bar v \rangle, \\
\end{cases}
\end{equation}
here $Q= \frac{Du}{m^{\alpha}}$ and the Hamiltonian $H_{\lambda}$ and its derivative are evaluated at the point $(x, Q).$  
\vskip0.3cm

\item $f(x,0)=A(x), v(x,T_0)=B(x)$.

\end{enumerate}
 \end{df}
 	
\begin{rem}
Note that \eqref{weaksol}  implies $v,f\in C(\mathcal{T}; L^2(\td))$( see e.g. \cite{E6}, Section 5.9.2, Theorem 3). Therefore, 
 the traces $f(x,0), v(x,T_0)$ are well-defined.
\end{rem}	
	 
\begin{teo}[Uniqueness of weak solutions]\label{weakuniq}
Let $(u_{\lambda},m_{\lambda})$ be a $C^\infty$ solution to \eqref{maineqlambda}, and let $T_0$ be as in Theorem \ref{bootstrapp}. Then, under Assumptions \ref{psi}--\ref{uniq3}, there exists at most one weak solution to the equation $\mathcal{L}_{\lambda}w=W$ in the sense of Definition \ref{weaksolu}.
\end{teo}

\begin{proof}
Since the equation $\mathcal{L}_{\lambda}w=W$ is linear, it is enough to prove that $\mathcal{L}_{\lambda}w=0$ has only the trivial solution $w=0$. For this, we take $\bar f=v, \bar v= f$ in \eqref{weak}. Adding both 
equations and integrating in time, we obtain
\[\begin{split}
0=&\int_0^{T_0}\int_{\td}\Big [ \alpha m^{\alpha-1}f^2\left(Q\cdot D_pH_{\lambda}-H_{\lambda}\right)+m^{1-\alpha}  Dv\cdot D^2_{pp}H_{\lambda}\cdot  Dv\\&- \alpha f Q\cdot D^2_{pp}H_{\lambda}\cdot Dv+D_zV_{\lambda}f^2\Big ]  dx dt =
\\&
\int_0^{T_0}\int_{\td}\Big [\alpha m^{\alpha-1}f^2\left(Q\cdot D_pH_{\lambda}-H_{\lambda}-\frac{\alpha } {4}Q\cdot D^2_{pp}H_{\lambda}\cdot Q\right)\\&+m^{\alpha-1}\left(m^{1-\alpha}Dv-
\frac{\alpha} {2}f Q\right)^t\cdot D^2_{pp}H_{\lambda}\cdot\left(m^{1-\alpha}Dv-
\frac{\alpha} {2}f Q\right)+D_zV_{\lambda}f^2\Big ]  dx dt, 
\end{split}
\]
where we set $Q=\frac{Du}{m^{\alpha}}$. Using the estimates from Theorem \ref{1minf2}, Lemma \ref{ulip}, Remark \ref{rem4}, and Assumption \ref{uniq3}, we conclude that at a solution $(u_{\lambda},m_{\lambda})$ to \eqref{maineqlambda} there exist constants $\theta_1, \theta_2>0$ that do not depend on the solution and $\lambda$, such that the above expression bounded by below by
\[\begin{split}
\theta_1 \int_0^{T_0}\int_{\td}m^{\alpha-1}\left|m^{1-\alpha}Dv-
\frac{\alpha} {2}f Q\right|^2+\theta_2 |f|^2\, dx dt.
\end{split}
\]
Thus, we get  $f=0, Dv=0$.
Consequently $v\equiv v(t)$. Next, by looking at the second equation in \eqref{weak}, 
for $\bar v=v(t)$ and $\bar f=0$, we obtain
\[
\frac{d}{dt}\langle v, v \rangle=0. 
\]
Using the boundary conditions for $v$, we conclude that $v=0$. 
Therefore, $w=0$.
\end{proof}

 To prove the existence of weak solutions, we apply the Galerkin approximation method (see e.g. \cite{E6}). We consider  a sequence of $C^\infty$ functions $e_k=e_k(x),\ k\in \Nn$ such that $\{\ e_k\ \}_{k=1}^{\infty}$ is an orthogonal basis of  $H^1(\td)$ and an orthonormal basis of $L^2(\td)$. We construct a sequence of finite dimensional approximations
 to 
  weak solutions
  of \eqref{maineqlambda}
  as follows, let $v_N, f_N \colon [0,T_0]\to H^1(\td)$
 \[
 f_N(t)=\sum_{k=1}^N  A^k_N(t) e_k,\quad v_N(t)=\sum_{k=1}^N  B^k_N(t) e_k.
 \]
We will show that we can  select the coefficients $A^k_n, B^k_N$ so that 
\begin{equation}\label{Galerk}
\begin{cases}
\langle {f'_N}, e_k\rangle +\left\langle Df_N +f_N D_pH_{\lambda}+m^{1-\alpha}D^2_{pp}H_{\lambda}\cdot Dv_N
-\alpha f_N D^2_{pp}H_{\lambda}\cdot Q +f_N b_{\lambda}, De_k \right\rangle\\
\qquad = \langle h , e_k\rangle,\\\\

\langle {v'_N}, e_k \rangle -\langle D v _N, De_k \rangle \\
\qquad -\left\langle \alpha m^{\alpha-1}f_N \left( H_{\lambda} -Q\cdot D_pH_{\lambda}\right)+D_pH_{\lambda} \cdot Dv_N +b_{\lambda}\cdot Dv_N-D_zV_{\lambda}f_N, e_k\right\rangle= \langle g , e_k \rangle
\end{cases}
\end{equation}
and
\begin{equation}
\label{itbc}
A^k_N(0)=\langle A, e_k \rangle,\ B^k_N(T_0)=\langle B, e_k \rangle,\qquad k=1, 2, \ldots, N.
\end{equation}
The  system \eqref{Galerk} is equivalent to:
\begin{equation}\label{Galerkin2}
\begin{cases}
\dot A_N^k +\sum_{l=1}^N\left\langle De_l + e_l D_pH_{\lambda}
-\alpha e_l D^2_{pp}H_{\lambda}\cdot Q +e_l b_{\lambda}, De_k \right\rangle A_N^l\\\qquad +\sum_{l=1}^N\langle m^{1-\alpha}D^2_{pp}H_{\lambda}\cdot De_l,De_k\rangle B_N^l= \langle h , e_k\rangle,\\\\
\dot B_N^k -\sum_{l=1}^N \langle D e_l+D_pH_{\lambda} \cdot De_l +b_{\lambda}\cdot De_l, De_k \rangle B^l_N\\
\qquad - \sum_{l=1}^N \left\langle \alpha m^{\alpha-1}e_l \left( H_{\lambda} -Q\cdot D_pH_{\lambda}\right)-e_lD_zV_{\lambda}, e_k\right\rangle A^l_N= \langle g , e_k \rangle.
\end{cases}
\end{equation}

Because \eqref{Galerkin2} is a linear system of ordinary differential equations, 
the only difficulty in proving the existence of solutions 
concerns the boundary conditions \eqref{itbc}. Existence is not 
immediate because
half of the boundary conditions are given at
the initial time, whereas the other half are given at the terminal time. 
From  standard theory of ordinary differential equations, the initial value problem
for \eqref{Galerkin2}, that is, with $A^k_N(0)$ and $B^k_N(0)$ prescribed, 
has a unique solution.
Hence, 
to prove the existence of solutions to \eqref{Galerkin2}, it is enough to 
show the existence of solutions for the corresponding homogeneous problem:
\begin{equation}\label{Galerkin3}
\begin{cases}
\dot {\tilde A}_n^k +\sum_{l=1}^N\left\langle De_l + e_l D_pH_{\lambda}
-\alpha e_l D^2_{pp}H_{\lambda}\cdot Q +e_l b_{\lambda}, De_k \right\rangle \tilde A_N^l\\\qquad +\sum_{l=1}^N\langle m^{1-\alpha}D^2_{pp}H_{\lambda}\cdot De_l,De_k\rangle \tilde B_N^l= 0\\\\
\dot {\tilde B}_N^k -\sum_{l=1}^N \langle D e_l+D_pH_{\lambda} \cdot De_l +b_{\lambda}\cdot De_l, De_k \rangle \tilde B^l_N\\
\qquad - \sum_{l=1}^N \left\langle \alpha m^{\alpha-1}e_l \left( H_{\lambda} -Q\cdot D_pH_{\lambda}\right)-e_lD_zV_{\lambda}, e_k\right\rangle \tilde A^l_N=0,\\
\end{cases}
\end{equation}
with arbitrary $\tilde A^k_N(0)$ and $\tilde B^k_N(T_0)$, $1\leq k\leq N$. 
Indeed, any solution to \eqref{Galerkin2}-\eqref{itbc}, $(A,B)$
 can be written as a sum of a particular solution to \eqref{Galerkin2}, $(\bar A,\bar B)$,
 for instance with
 \[
\bar A^k_N(0)=0,\  \bar B^k_N(0)=0,\ k=1, 2, \ldots, N.
 \]
with a solution, $(\tilde A, \tilde B)$ to \eqref{Galerkin3} with suitable initial and terminal conditions so that \eqref{itbc} holds for $(A,B)=(\bar A+\tilde A, \bar B+\tilde B)$.  

Next, we regard the solution of the initial value problem for the homogeneous system corresponding to \eqref{Galerkin2} as a
linear operator on $\Rr^{2N}$:
\begin{equation}
\label{map}
(A_N(0), B_N(0))\mapsto (A_N(0), B_N(T_0)).
\end{equation}
We need to prove that this mapping is surjective. Since \eqref{map} is a linear mapping from $\Rr^{2N}$ to $\Rr^{2N}$, surjectivity is equivalent to injectivity. Therefore, it suffices to prove that the homogeneous system of ODE's corresponding to \eqref{Galerkin2} subject to initial-terminal conditions $A_N(0)=B_N(T_0)=0$ has only the trivial solution $A_N=B_N\equiv 0$.
Let $f_N, v_N$ solve \eqref{Galerk} with $h=g\equiv 0, A=B\equiv 0$. From \eqref{Galerk}, we obtain \eqref{weak} for $f=\bar v=f_N, v=\bar f= v_N$. Using the same argument as in  Theorem \ref{weakuniq}, we conclude that $f_N=v_N\equiv 0$.

Next, we prove energy estimates for these approximations to
ensure the weak convergence of approximate solutions through some subsequence. 
\begin{teo}
Suppose Assumptions \ref{psi}--\ref{uniq3} hold. Then, for $T_0$ small enough, there exists a constant $C$ such that
for any   $C^\infty$ solution $(u_{\lambda},m_{\lambda})$ to \eqref{maineqlambda}, we have
\[
\begin{split}
&\max_{0\leq t\leq T_0} \|(f_N,v_N)\|_{(L^2(\td))^2}+ \|(f_N,v_N)\|_{
	(L^2(0,T_0; H^1(\td)))^2}+ \|(f'_N,v'_N)\|_{(L^2(0,T_0; H^{-1}(\td)))^2}\\
&\qquad \leq C\left( \|h\|_{L^2(0,T_0; L^2(\td))}+\|g\|_{L^2(0,T_0; L^2(\td))}+ \|A\|_{L^2(\td)}+ \|B\|_{L^2(\td)}\right).
\end{split}
\]
\end{teo}

\begin{proof}
 We assume $T_0$ is small enough so that Theorem \ref{bootstrapp} holds. Using the linearity of \eqref{Galerk}, we observe that \eqref{weak} holds
for $f=\bar f=f_N, v_N=\bar v_N=v_N$. Then, using H\"older's inequality and the estimates from Theorem \ref{bootstrapp}, we obtain
the system of inequalities:
\begin{equation}\label{ineq}
\begin{cases}
\left(  \|f_N\|^2_{L^2(\td)} \right)_t+\|Df_N\|^2_{L^2(\td)}\leq C\left( \|h\|^2_{L^2(\td)} 
+ \|Dv_N\|^2_{L^2(\td)}  + \|f_N\|^2_{L^2(\td)} \right),\\
\left(  \|v_N\|^2_{L^2(\td)} \right)_t-\|Dv_N\|^2_{L^2(\td)}\geq -C\left( \|g\|^2_{L^2(\td)} + \|v_N\|^2_{L^2(\td)} + \|f_N\|^2_{L^2(\td)} \right).
\end{cases}
\end{equation}
From the second inequality, using Gronwall's inequality, we get 
\[
\|v_N(\cdot, t)\|^2_{L^2(\td)} \leq C \int_t^{T_0} \left( \|g(\cdot, s)\|^2_{L^2(\td)} + \|f_N(\cdot, s)\|^2_{L^2(\td)} \right)ds + C \|B\|^2_{L^2(\td)},
\]
and further
\[
\int_0^{T_0}\|Dv_N(\cdot, s)\|^2_{L^2(\td)}ds \leq C \int_0^{T_0} \left( \|g(\cdot, s)\|^2_{L^2(\td)} + \|f_N(\cdot, s)\|^2_{L^2(\td)} \right)ds + C \|B\|^2_{L^2(\td)}.
\]

From this, combined with the first inequality in \eqref{ineq}, and using Gronwall's inequality once more,we have
\[
\|f_N(\cdot, t)\|^2_{L^2(\td)} \leq C \int_0^{T_0} \left( \|g(\cdot, s)\|^2_{L^2(\td)} + \|h(\cdot, s)\|^2_{L^2(\td)} + \|f_N(\cdot, s)\|^2_{L^2(\td)} \right)ds + C\left( \|A\|^2_{L^2(\td)}+\|B\|^2_{L^2(\td)}\right). 
\]
Thus, for $T_0$ small enough, we get
\[
\sup_{t\in[0,T_0]}\|f_N(\cdot, t)\|^2_{L^2(\td)} \leq C \int_0^{T_0} \left( \|g(\cdot, s)\|^2_{L^2(\td)} + \|h(\cdot, s)\|^2_{L^2(\td)} \right)ds + C\left( \|A\|^2_{L^2(\td)}+\|B\|^2_{L^2(\td)}\right). 
\]
Consequently,
\[
\int_0^{T_0}\|Df_N(\cdot, t)\|^2_{L^2(\td)} \leq C \int_0^{T_0} \left( \|g(\cdot, s)\|^2_{L^2(\td)} + \|h(\cdot, s)\|^2_{L^2(\td)} \right)ds + C\left( \|A\|^2_{L^2(\td)}+\|B\|^2_{L^2(\td)}\right).
\]
Thus, we have
\[
\begin{split}
&\max_{0\leq t\leq T_0} \|(f_N,v_N)\|_{(L^2(\td))^2}+ \|(f_N,v_N)\|_{(L^2(0,T_0; H^1(\td)))}\\
&\leq C\left( \|h\|_{L^2(0,T_0; L^2(\td))}+\|g\|_{L^2(0,T_0; L^2(\td))}+ \|A\|_{L^2(\td)}+ \|B\|_{L^2(\td)}\right).
\end{split}
\]
From equation \eqref{weak}, for any $\bar f, \bar v\in span \{e_k\}_{k=1}^N$
with $\|\bar f\|_{L^2(0,T_0; H^1(\td))}\leq 1$, $\|\bar v\|_{L^2(0,T_0; H^1(\td))}\leq 1$ we get
\[
\begin{cases}
\int_0^{T_0} \langle f_N(s), \bar f\rangle ds\leq C\left( \|h\|^2_{L^2(0,T_0; L^2(\td))}+\|g\|^2_{L^2(0,T_0; L^2(\td))}+ \|A\|^2_{L^2(\td)}+ \|B\|^2_{L^2(\td)}\right),\\
\int_0^{T_0} \langle v_N(s), \bar v\rangle ds\leq C\left( \|h\|^2_{L^2(0,T_0; L^2(\td))}+\|g\|^2_{L^2(0,T_0; L^2(\td))}+ \|A\|^2_{L^2(\td)}+ \|B\|^2_{L^2(\td)}\right),
\end{cases}
\]
since $f_N, v_N\in span \{e_k\}_{k=1}^N$ as well. These inequalities imply the required estimates.
\end{proof}

\begin{teo}[Existence of weak solutions]
\label{t2}
Let $(u_{\lambda},m_{\lambda})$ be a $C^\infty$ solution to  \eqref{maineqlambda} and let $T_0$ be as in Theorem \ref{bootstrapp}. Then, under Assumptions \ref{psi}--\ref{uniq3}, there exists a weak solution to the equation $\mathcal{L}_{\lambda}w=W$ in the sense of \eqref{weak}.
\end{teo}

\begin{proof}
According to the energy estimates, there exist subsequences of $v_{N}, f_N$ and functions $v,f \in L^2(0,T_0; H^1(\td)),$ with  $v'=v_t, f'=f_t \in L^2(0,T_0; H^{-1}(\td)),$ such that
\[
\begin{cases}
v_N \rightharpoonup v,\ f_N \rightharpoonup f, \text{ weakly in  }\  L^2(0,T_0; H^1(\td))\\
v'_N \rightharpoonup v',\ f'_N \rightharpoonup f', \text{ weakly in  }\  L^2(0,T_0; H^{-1}(\td)).
\end{cases}
\]
For fixed $N_0$, let $\bar v, \bar f \in span \{e_k\}_{k=1}^{N_0}$ with $\|\bar v\|_{L^2(0,T_0; H^1(\td))},\| \bar f \|_{L^2(0,T_0; H^1(\td))}\leq 1$. According to the definition of $v_N, f_N$, we have that \eqref{weak} holds for every $N\geq N_0$. Weak convergence then implies \eqref{weak} for $v, f$ and any $\bar v, \bar f \in span \{e_k\}_{k=1}^{N_0}$. The above convergence implies that $v_N \rightharpoonup v,\ f_N \rightharpoonup f$ also in $C(0,T_0; L^{2}(\td))$.
Therefore, the the initial and terminal conditions on $f,v$ hold as well. Since $\cup_{N\geq 1} span \{e_k\}_{k=1}^{N} $ is dense in $L^2(0,T_0; H^1(\td))$, we are done.
\end{proof}

\begin{teo}[Higher Regularity]
\label{t3}
Let $(u_{\lambda},m_{\lambda})$ be a $C^\infty$ solution of \eqref{maineqlambda} and let $T_0$ be as in Theorem \ref{bootstrapp}. Assume $A,B \in H^{k+1}(\td), h,g \in F^ {2k}(\mathcal{T}; \td)$ and let $W=[h, g, A, B]^t$. Then, under Assumptions \ref{psi}--\ref{uniq3}, for any weak solution $w=[f, v]^t$ of  $\mathcal{L}_{\lambda}w=W$, we have $v,f \in F^ {2k+2}(\mathcal{T}; \td)$. 
\end{teo}

The proof draws on the regularizing properties of the heat equation and a bootstrap argument. We use the following result:
\begin{lemma}
	\label{L6}
Let $\tilde h \in H^{k_1}(\mathcal{T}; H^{k_2}(\td))$, $\tilde g \in H^{2k_1+k_2+1} (\td)$ for some $k_1, k_2\geq 0$, and let $\tilde u\in F^{1}(\mathcal{T}; \td)$ be a weak solution of the heat equation
\[
\begin{cases}
\tilde u_t-\Delta \tilde u=\tilde h \\
\tilde u(x,0)=\tilde g(x).
\end{cases}
\]
Then $\tilde u\in H^{k_1}(\mathcal{T}; H^{k_2+2}(\td))\cap H^{k_1+1}(\mathcal{T}; H^{k_2}(\td)).$
\end{lemma}

\begin{proof}
The Lemma is proved easily using induction. The base case $k_1=k_2=0$ is a standard regularity result for the heat equation.
\end{proof}

From the second equation of \eqref{weak}, we have that $v$ is a weak solution to
\begin{equation}
\label{efv}
\begin{cases}
v_t+\Delta v=g
\\
\quad +\alpha m^{\alpha-1}f \left( H_{\lambda}\left(x, \frac{Du}{m^{\alpha}}\right) + \frac{Du}{m^{\alpha}} D_pH_{\lambda}\left(x, \frac{Du}{m^{\alpha}}\right)\right)+D_pH_{\lambda}\left(x, \frac{Du}{m^{\alpha}}\right)Dv+b_{\lambda}\cdot Dv+D_zV_{\lambda}f\\\\
v(x,T_0)=B(x).
\end{cases}
\end{equation}
Since the right-hand side of the previous PDE belongs to $L^2(0,T_0, L^2(\td))$, using Lemma \ref{L6}, we conclude that 
 $v \in L^2(\mathcal{T}; H^{2}(\td))\cap H^{1}(\mathcal{T}; L^2(\td))$. 

Next, the first equation of \eqref{weak} implies that $f$ is a weak solution to
\begin{equation}
\label{eff}
\begin{cases}
f_t-\Delta f=h \\
\quad +\div\left[D_pH_{\lambda}\left(x, \frac{Du}{m^{\alpha}}\right)f
-\alpha f D^2_{pp}H_{\lambda}\left(x, \frac{Du}{m^{\alpha}}\right)\cdot  \frac{Du}{m^{\alpha}}+m^{1-\alpha}D^2_{pp}H_{\lambda}\left(x, \frac{Du}{m^{\alpha}}\right)\cdot Dv +b_{\lambda}f\right] \\\\
f(x,0)=A(x).
\end{cases}
\end{equation}
From the regularity of $v$ obtained above, we conclude that the right-hand side of this equation is also in $L^2(0,T_0, L^2(\td))$. For that reason, according to Lemma \ref{L6}, $f  \in L^2(\mathcal{T}; H^{2}(\td))\cap H^{1}(\mathcal{T}; L^2(\td))$.
	
Now, we assume $v,f\in F^ {2i}(\mathcal{T}; \td)$ for some $i\leq k$, we will prove that  $v,f \in F^ {2i+2}(\mathcal{T}; \td)$.
First, note that since $v,f\in H^{k_1}(\mathcal{T}; H^{k_2}(\td))$ for every $k_1, k_2$ with $2k_1+k_2= 2i$, the expression on the right-hand side of \eqref{efv} is in $H^{k_1}(\mathcal{T}; H^{k_2-1}(\td))$. Thus, using Lemma \ref{L6},
we get $v\in H^{k_1}(\mathcal{T}; H^{k_2+1}(\td))$. We know now that
the right-hand side of \eqref{efv} is in $H^{k_1}(\mathcal{T}; H^{k_2}(\td))$.	
Using Lemma \ref{L6} the second time, we conclude that $v\in H^{k_1}(\mathcal{T}; H^{k_2+2}(\td))\cap H^{k_1+1}(\mathcal{T}; H^{k_2}(\td))$.

Now, we have that the right-hand side of \eqref{eff} is in  $H^{k_1}(\mathcal{T}; H^{k_2-1}(\td))$. Thus, using Lemma \ref{L6} again twice as above, we get $f\in H^{k_1}(\mathcal{T}; H^{k_2+2}(\td))\cap H^{k_1+1}(\mathcal{T}; H^{k_2}(\td))$. From what we have proved, it follows $v,f\in H^{\tilde k_1}(\mathcal{T}; H^{\tilde k_2}(\td))$, for every $\tilde k_1,\tilde k_2$ with $2\tilde k_1+\tilde k_2= 2i+2$. Consequently,
 $v,f \in F^ {i+2}(\mathcal{T}; \td)$.


\begin{proof}[Proof of the Theorem \ref{main}]

Theorem \ref{bootstrapp} and Arzela-Ascoli Theorem imply that the set $\Lambda$ is a closed subset of the interval $[0,1]$. We will prove that it is also open. Let $\lambda_0\in \Lambda$.  Using Theorem \ref{bootstrapp}, we see that the operator 
\[
\mathcal{L}_{\lambda_0}\colon F^{2k}(\mathcal{T}; \td)\times F^{2k}(\mathcal{T}; \td)\to
F^{2k-2}(\mathcal{T}; \td)\times F^{2k-2}(\mathcal{T}; \td)\times H^{2k-1}(\td)\times H^{2k-1}(\td)
\]
is bounded for every $k\geq 1$. Using Theorems \ref{weakuniq},  \ref{t2}, and \ref{t3}, we conclude that $\mathcal{L}_{\lambda_0}$ is bijective, and so it is invertible. 
We choose $k$ large enough so that $H^l(\mathcal{T}; H^l(\td))$, where $l=\lfloor \frac {2k} 3 \rfloor$, is an algebra. By the inverse function theorem (\cite{D1}), there is a neighborhood $U$ of $\lambda_0$ where the equation $\mathcal{M}_{\lambda}\left[\begin{array}{c}
u\\
m
\end{array}\right]=0$ has a unique solution $(u_{\lambda}, m_{\lambda})$ in $F^{2k}(\mathcal{T}; \td)\times F^{2k}(\mathcal{T}; \td)$. Then, $u_{\lambda}, m_{\lambda}\in H^l(\mathcal{T}; H^l(\td))$. The inverse function theorem implies that the mapping $\lambda\mapsto(u_{\lambda}, m_{\lambda})$ is continuous. Hence, we can assume that in the neighborhood $U$, $m_{\lambda}$ is bounded away from zero. This observation, together with the fact that $H^l(\mathcal{T}; H^l(\td))$ is an algebra allows us to use regularity theory and bootstrap arguments to conclude that $(u_{\lambda}, m_{\lambda})$ are $C^\infty$. Accordingly, $U\subset \Lambda$.  Consequently, we have 
proved that $\Lambda$ is an open set in $[0,1]$. Because $1\in \Lambda$, we
know that $\Lambda\neq\emptyset$. Therefore, $\Lambda=[0,1]$. In particular, $0\in \Lambda$. \end{proof}

\bibliographystyle{plain}

\bibliography{mfg}

\end{document}